\documentclass[reqno,b5paper]{amsart}
\usepackage{amsmath}
\usepackage{amssymb}
\usepackage{amsthm}

\newtheorem{prop}{Proposition}[section]
\newtheorem{lem}{Lemma}[section]
\newtheorem{coro}{Corollary}[section]
\theoremstyle{definition}
\newtheorem{defi}{Definition}[section]
\newtheorem{exa}{Example}[section]
\newtheorem*{rem}{Remark}

\begin{document}

\title{Tying up baric algebras}

\author[Antonio M. Oller-Marc\'{e}n]{Antonio M. Oller-Marc\'{e}n*}

\newcommand{\acr}{\newline\indent}

\address{\llap{*\,}Departamento de Matem\'{a}ticas\acr
Universidad de Zaragoza\acr
C/ Pedro Cerbuna 12, 50009\acr
Zaragoza\acr
SPAIN}
\email{oller@unizar.es} 

\subjclass[2010]{Primary 17D92; Secondary 17D99}

\keywords{Baric algebra, Indecomposable baric algebra}

\begin{abstract}
Given two baric algebras $(A_1,\omega_1)$ and $(A_2,\omega_2)$ we
describe a way to define a new baric algebra structure over the
vector space $A_1\oplus A_2$, which we shall denote $(A_1\bowtie
A_2,\omega_1\bowtie\omega_2)$. We present some easy properties of this construction and we show that in the commutative and unital case it preserves indecomposability. Algebras of the form $A_1\bowtie A_2$ in the associative, coutable-dimensional, zero-characteristic case are classified.  
\end{abstract}

\maketitle 

\section{Introduction}
Baric algebras play an important role in the theory of genetic
algebras. The use of algebraic formalism to study genetic
inheritance was introduced by I.M.H. Etherington \cite{ETH} in the
first half of the last century and has revealed fruitful giving rise
to many interesting classes of algebras such as train or Bernstein.
For a brief survey of this subject we refer to \cite{MLR} and for an
introductory but deeper approach to \cite{ANG}.

In \cite{CG} the notion of decomposable baric algebras was introduced. In the same paper it was also presented a way to construct decomposable baric algebras starting from two baric algebras with an idempotent of weight one. Furthermore in \cite{CG} and in \cite{CGb} the authors analized the indecomposability of some well-known examples of algebras arising in genetics. In this work we define a new way to construct a baric algebra starting from two given baric algebras. Our construction, although similar, is different than that in \cite{CG}. In particular, while the construction in \cite{CG} always gives rise to decomposable baric algebras, we will show that in the commutative unital case our construction preserves indecomposability.

We will also show that baric algebras obtained by our method always have a unique weight homomorphism. Thus, as a consequence, we show that every baric algebra can be embedded in a baric algebra with a unique weight homomorphism.

The paper is organized as follows. The second section presents the construction and the third one gives some properties following easily frome the definition. In the fourth section we study the uniqueness of the weight homomorphism. In the fifth section we study the ideals and
focus on the case when our original algebras are commutative and unital, showing that in this case our construction preserves indecomposability. Finally we study the associative case and give a classification when we are in countable dimension and the base field is of characteristic zero.

\section{The construction}
Let $(A_1,\omega_1)$ and $(A_2,\omega_2)$ be two baric algebras;
i.e, $A_1$ and $A_2$ are algebras over a field $K$ and
$\omega_i:A_i\longrightarrow K$ is a non-zero $K$-algebra
homomorphism for $i=1,2$. Now, in the $K$-vector space $A_1\oplus
A_2$ we define a product
\begin{equation}\label{uno}(a_1,a_2)(b_1,b_2)=(a_1b_1+\omega_2(b_2)a_1,a_2b_2+\omega_1(b_1)a_2)\end{equation}
which is easily seen to define a $K$-algebra structure on $A_1\oplus
A_2$.

\begin{defi}
Given $(A_i,\omega_i)$ with $i=1,2$ two baric algebras we define
$A_1\bowtie A_2$ to be the $K$-vector space $A_1\oplus A_2$ with the
algebra structure given by the product (\ref{uno}).
\end{defi}

We can now define an application $\omega_1\bowtie\omega_2:A_1\oplus
A_2\longrightarrow K$ given by the formula
$\omega_1\bowtie\omega_2(a_1,a_2)=\omega_1(a_1)+\omega_2(a_2)$.
Trivially $\omega_1\bowtie\omega_2$ is $K$-linear and, also, we have
that
\begin{align*}\omega_1\bowtie\omega_2((a_1,a_2)(b_1,b_2))&=\omega_1\bowtie\omega_2(a_1b_1+\omega_2(b_2)a_1,a_2b_2+\omega_1(b_1)a_2)\\&=\omega_1(a_1b_1)+\omega_2(b_2)\omega_1(a_1)+\omega_2(a_2b_2)+\omega_1(b_1)\omega_2(a_2)\\&=(\omega_1(a_1)+\omega_2(a_2))(\omega_1(b_1)+\omega_2(b_2))\\&=\omega_1\bowtie\omega_2(a_1,a_2)\omega_1\bowtie\omega_2(b_1,b_2).\end{align*}
Thus, $\omega_1\bowtie\omega_2$ is a $K$-homomorphism and we have
the following:

\begin{prop}
Let $(A_i,\omega_i)$ with $i=1,2$ be baric algebras. Then so is the
pair $(A_1\bowtie A_2,\omega_1\bowtie\omega_2)$.
\end{prop}

\begin{rem}
By means of the inclusions $\iota_i:A_i\hookrightarrow A_1\bowtie
A_2$ ($i=1,2$) given by $\iota_i(a_i)=(\delta_i^1a_i,\delta_i^2a_i)$
we can see each $A_i$ as a subalgebra of $A_1\bowtie A_2$ and we
will identify $A_i$ with $\iota_i(A_i)$. With this identification it
is easy to see that $A_i\unlhd_r A_1\bowtie A_2$.
\end{rem}

\begin{exa}
Let $K$ be a field. Obviously $(K,\textrm{id}_K)$ is a baric
algebra, then $K\bowtie K$ is the vector space $K^2$ endowed with
the product
$$(\alpha,\beta)(\alpha',\beta')=(\alpha'+\beta')(\alpha,\beta).$$
In this case we have
$$\textrm{id}_K\bowtie\textrm{id}_K(\alpha,\beta)=\alpha+\beta.$$ We
will come back to this example later on.
\end{exa}

\section{Some easy properties}
This section is devoted to present some properties arising easily from the previous construction.

We recall that two baric algebras $(A,\omega)$ and $(B,\varphi)$ are
said to be isomorphic if there exists a $K$-algebra isomorphism
$f:A\longrightarrow B$ such that $\varphi\circ f=\omega$. The
following propositions show some nice properties of this
construction.

\begin{prop}
Let $(A_i,\omega_i)$ with $i=1,2,3$ be baric algebras. Then we have
the following isomorphisms:
\begin{itemize}
\item[(i)] $(A_1\bowtie A_2,\omega_1\bowtie\omega_2)\cong(A_2\bowtie
A_1,\omega_2\bowtie\omega_1)$.
\item[(ii)] $((A_1\bowtie A_2)\bowtie
A_3,(\omega_1\bowtie\omega_2)\bowtie\omega_3)\cong(A_1\bowtie(A_2\bowtie
A_3),\omega_1\bowtie(\omega_2\bowtie\omega_3))$.
\end{itemize}
\end{prop}
\begin{proof}
Define $f_1:A_1\bowtie A_2\longrightarrow A_2\bowtie A_1$ by
$f(a_1,a_2)=(a_2,a_1)$, in the same way, define $f_2:(A_1\bowtie
A_2)\bowtie A_3\longrightarrow A_1\bowtie(A_2\bowtie A_3)$ by
$f_2((a_1,a_2),a_3)=(a_1,(a_2,a_3))$. It is easy to see that both
maps are weight-preserving $K$-isomorphisms.
\end{proof}

\begin{prop}
Let $(A_1,\omega_1)$, $(A'_1,\omega'_1)$ and $(A_2,\omega_2)$ be
baric algebras and let us suppose that
$(A_1,\omega_1)\cong(A'_1,\omega'_1)$, then $(A_1\bowtie
A_2,\omega_1\bowtie\omega_2)\cong(A'_1\bowtie
A_2,\omega'_1\bowtie\omega_2)$.
\end{prop}
\begin{proof}
We know by hypothesis that there exists an isomorphism
$f:A_1\longrightarrow A'_1$ such that $\omega'_1\circ f=\omega_1$.
We can define a map $\widetilde{f}:A_1\bowtie A_2\longrightarrow
A'_1\bowtie A_2$ in a natural way by
$\widetilde{f}(a_1,a_2)=(f(a_1),a_2)$. This map is obviously a
$K$-homomorphism and, moreover,
$\omega'_1\bowtie\omega_2(\widetilde{f}(a_1,a_2))=\omega'_1(f(a_1))+\omega_2(a_2)=\omega_1(a_1)+\omega_2(a_2)=\omega_1\bowtie\omega_2(a_1,a_2)$
and this completes the proof.
\end{proof}

Given a $K$-algebra $A$ and elements $x,y,z\in A$, the definitions
of the commutator $[x,y]=xy-yx$ and of the associator
$(x,y,z)=(xy)z-x(yz)$ are well known; $A$ being commutative or
associative if and only if $[x,y]=0$ for all $x,y\in A$ or
$(x,y,z)=0$ for all $x,y,z\in A$ respectively.

\begin{lem}
Let $(A_i,\omega_i)$ for $i=1,2$ be baric algebras. Let
$x=(a_1,a_2)$ and $y=(b_1,b_2)$ be elements of $A_1\bowtie A_2$.
Then we have that:
$$[x,y]=([a_1,b_1]+\omega_2(b_2)a_1-\omega_2(a_2)b_1,[a_2,b_2]+\omega_1(b_1)a_2-\omega_1(a_1)b_2).$$
\end{lem}

Recall that the commutative center of an algebra $A$ is the set
$$\mathcal{K}(A)=\{a\in A\ |\ [a,b]=0\ \forall b\in A\}$$
In view of the previous lemma, we have the following corollary.

\begin{coro}
If $(A_i,\omega_i)$ with $i=1,2$ are baric algebras, then
$\mathcal{K}(A_1\bowtie A_2)=0$
\end{coro}

As usual it is interesting to search for idempotents. In the case of
baric algebras we look for idempotents of weight 1. In this
direction we have the following easy result.

\begin{prop}
Let $(A_i,\omega_i)$ for $i=1,2$ be baric $K$-algebras and let
$e_i\in A_i$ be idempotents such that $\omega_i(e_i)=1$. Consider
the set $\mathfrak{I}=\{(\lambda e_1,\mu e_2)\ |\ \lambda+\mu=1\}$,
then $ef=e$ for all $e,f\in\mathfrak{I}$ and, in particular,
$\mathfrak{I}$ consists of idempotents of weight 1.
\end{prop}
\begin{proof}
$(\lambda_1 e_1,\mu_1 e_2)(\lambda_2 e_1,\mu_2
e_2)=(\lambda_1(\lambda_2+\mu_2)e_1,\mu_1(\lambda_2+\mu_2)e_2)$
\end{proof}

\section{Uniqueness of the weight homomorphism}
In a baric algebra, the weight homomorphism is not uniquely
determined in general (see \cite{LSR} for example). Nevertheless,
the following result shows that our construction behaves quite
nicely in this sense.

\begin{prop}
Let $(A_i,\omega_i)$ for $i=1,2$ be baric algebras. Then
$\omega_1\bowtie\omega_2$ is uniquely determined.
\end{prop}
\begin{proof}
Let us suppose that $\varphi:A_1\bowtie A_2\longrightarrow K$ is a
non-trivial homomorphism of $K$-algebras. Then for $i=1,2$ we can
define $\varphi_i:A_i\longrightarrow K$ by
$\varphi_1(a_1)=\varphi(a_1,0)$ and $\varphi_2(a_2)=\varphi(0,a_2)$.
It is easy to check that both $\varphi_i$ are $K$-homomorphisms and
$\varphi(a_1,a_2)=\varphi_1(a_1)+\varphi_2(a_2)$.

Now, as $\varphi$ is a $K$-homomorphism,
$\varphi(a_1,a_2)\varphi(b_1,b_2)=\varphi\big(
(a_1,a_2)(b_1,b_2)\big)$ and thus, by the preceding considerations:
$$(\varphi_1(a_1)+\varphi_2(a_2))(\varphi_1(b_1)+\varphi_2(b_2))=\varphi(a_1b_1+\omega_2(b_2)a_1,a_2b_2+\omega_1(b_1)a_2)$$
From this it follows that:
$$\varphi_1(a_1)(\omega_2(b_2)-\varphi_2(b_2))=\varphi_2(a_2)(\varphi_1(b_1)-\omega_1(b_1)),\
\forall a_i,b_i\in A_i$$ So choosing $a_1\in\textrm{Ker}\ \varphi_1$
and $a_2\notin\textrm{Ker}\ \varphi_2$ we have that
$\varphi_1=\omega_1$. Similarly we obtain $\varphi_2=\omega_2$ and
the proof is complete.
\end{proof}

As a consequence of this result, together with the fact that $A_i$ is a subalgebra of $A_1\bowtie A_2$ we have the following corollary.

\begin{coro}
Every baric algebra is a subalgebra of a baric algebra with a unique weight homomorphism.
\end{coro}

In \cite{ANG} it is shown that if a baric algebra $(A,\omega)$ is such that $\textrm{Ker}\ \omega$ is nil, then the weight homomorphism is uniquely determined. Clearly our construction provides a family of examples showing that the converse is false.

\section{Ideals and indecomposability.}
Let $(A_i,\omega_i)$ with $i=1,2$ be baric algebras, then each $A_i$
can be seen as a subalgebra of $A_1\bowtie A_2$. Now let $I\unlhd_r
A_1$ be a right ideal. We can identify $I$ with $\iota_1(I)$ and it
is easy to see that with this identification $I\unlhd_r A_1\bowtie
A_2$ remains a right ideal.

Now let $I\unlhd_r A_1\bowtie A_2$ be a right ideal. Then we can
define
$$I_1=\{a_1\in A_1\ |\ \exists\ a_2\in A_2\ \textrm{s.t.}\
(a_1,a_2)\in I\}$$ Again, it is easy to see that $I_1\unlhd_r A_1$
is also a right ideal. Note that if we define the projections
$p_i:A_1\bowtie A_2\longrightarrow A_i$ in the obvious way, $I_1$ is
just $p_1(I)$. In the same way we can define $I_2$.

In view of the previous considerations, it is natural to ask whether
an ideal of $A_i$ remains an ideal of $A_1\bowtie A_2$.

\begin{prop}
Let $I\unlhd A_i$ be an ideal. Then $I$ is an ideal of $A_1\bowtie
A_2$ if and only if $I\subseteq \textrm{Ker}\ \omega_i$.
\end{prop}
\begin{proof}
If $I\unlhd A_1$ (the case $I\unlhd A_2$ is analogous), then clearly
$I\unlhd_r A_1\bowtie A_2$. Now if $x\in I$ and $a_i\in A_i$ for
$i=1,2$ we have that $(a_1,a_2)(x,0)=(a_1x,\omega_1(x)a_2)\in I$ if
and only if $\omega_1(x)a_2=0$ for all $a_2\in A_2$. Obviously this
happens if and only if $\omega_1(x)=0$ and the proof is complete.
\end{proof}

While, on the other hand, we have the following:

\begin{prop}
Let $I\unlhd A_1\bowtie A_2$ be an ideal such that $I_1\neq A_1$.
Then $I_1$ is an ideal of $A_1$ if and only if $I_2\subseteq
\textrm{Ker}\ \omega_2$.
\end{prop}
\begin{proof}
Given $I\unlhd A_1\bowtie A_2$ we already know that $I_1\unlhd_r
A_1$ is a right ideal. Let us suppose that $I_2\subseteq
\textrm{Ker}\ \omega_2$, then if $a_1\in I_1$ and $a\in A_1$, there
exists $a_2\in A_2$ such that $(a_1,a_2)\in I$; so we have that
$(a,0)(a_1,a_2)=(aa_1,0)\in I$ and this implies that $aa_1\in I_1$
as desired.

Conversely, suppose that there exists $a_2\in I_2$ such that
$\omega_2(a_2)\neq 0$. By definition, there exists $a_1\in A_1$ such
that $(a_1,a_2)\in I$; in particular $a_1\in I_1$ so given any $a\in
A_1$ we have that $aa_1\in I_1$. Moreover,
$(a,0)(a_1,a_2)=(aa_1+\omega_2(a_2)a,0)\in I$ so
$aa_1+\omega_2(a_2)a\in I_1$. Then we have that $\omega_2(a_2)a\in
I_1$ and that $a\in I_1$. This implies $A_1=I_1$, a contradiction.
\end{proof}

\begin{rem}
If $I\unlhd A_1\bowtie A_2$ is an ideal, $I\neq A_1\bowtie A_2$ does
not imply $I_1\neq A_1$. To see this it is enough to consider the
ideal $I=\textrm{Ker}\ \omega_1\bowtie\omega_2$, in this case we
have that $I_1=A_1$ although $I\neq A_1\bowtie A_2$.
\end{rem}

\begin{prop}
Let $(A_i,\omega_i)$ for $i=1,2$ be commutative baric algebras and
let $I\unlhd A_1\bowtie A_2$ be an ideal such that
$I\subseteq\textrm{Ker}\ \omega_1\bowtie\omega_2$. Then $I_1=A_1$ if
and only if $I=\textrm{Ker}\ \omega_1\bowtie\omega_2$.
\end{prop}
\begin{proof}
Let us suppose $I_1=A_1$ and choose $a\in A_1$ such that
$\omega_1(a)\neq 0$. Then there exists $b\in A_2$ such that
$(a,b)\in I$, note that in particular $\omega_1(a)=-\omega_2(b)\neq
0$.

Now take $(a_1,a_2)\in\textrm{Ker}\ \omega_1\bowtie\omega_2$, i.e.,
$\omega_1(a_1)+\omega_2(a_2)=0$. Being $I\unlhd A_1\bowtie A_2$ and
due to the commutativity of each $A_i$ we have:
$$(a,b)(a_1,0)-(a_1,0)(a,b)=(-\omega_2(b)a_1,\omega_1(a_1)b)\in I$$
$$(a,b)(0,a_2)-(0,a_2)(a,b)=(\omega_2(a_2)a,-\omega_1(a)a_2)\in I$$
and there are two possible cases:

Firstly, if $\omega_1(a_1)=-\omega_2(a_2)=0$, then we have
$$(a_1,a_2)=-(\omega_2(b))^{-1}(-\omega_2(b)a_1,0)-(\omega_1(a))^{-1}(0,-\omega_1(a)a_2)\in
I$$ and secondly, if $\omega_1(a_1)=-\omega_2(a_2)\neq 0$, then
$$(a_1,a_2)=(\omega_2(b))^{-1}\Big((\omega_2(a_2)a,-\omega_1(a)a_2)-(-\omega_2(b)a_1,\omega_1(a_1)b)+\omega_1(a_1)(a,b)\Big)\in
I$$ Thus, in both cases $(a_1,a_2)\in I$ and the equality holds. The
converse was discussed in the previous remark.
\end{proof}

\begin{defi}
Let $(A,\omega)$ be a baric algebra. We define the set
$\mathcal{I}(A,\omega)$ to be:
$$\mathcal{I}(A,\omega)=\{I\unlhd A\ |\ I\subseteq\textrm{Ker}\
\omega\}$$
\end{defi}

\begin{prop}
Let $(A_i,\omega_i)$ for $i=1,2$ be commutative unital baric algebras. Then
the sets $\mathcal{I}(A_1,\omega_1)\times\mathcal{I}(A_2,\omega_2)$
and $\mathcal{I}(A_1\bowtie
A_2,\omega_1\bowtie\omega_2)\setminus\{\textrm{Ker}\
\omega_1\bowtie\omega_2\}$ are bijective.
\end{prop}
\begin{proof}
Let us define maps
$$\varphi:\mathcal{I}(A_1,\omega_1)\times\mathcal{I}(A_2,\omega_2)\longrightarrow\mathcal{I}(A_1\bowtie
A_2,\omega_1\bowtie\omega_2)\setminus\{\textrm{Ker}\
\omega_1\bowtie\omega_2\}$$ and $$\psi:\mathcal{I}(A_1\bowtie
A_2,\omega_1\bowtie\omega_2)\setminus\{\textrm{Ker}\
\omega_1\bowtie\omega_2\}\longrightarrow\mathcal{I}(A_1,\omega_1)\times\mathcal{I}(A_2,\omega_2)$$
by $\varphi(I,J)=I\bowtie J=\{(a,b)\in A_1\bowtie A_2\ |\ a\in I,\
b\in J\}$ and $\psi(I)=(I_1,I_2)$.

Proposition 5.1 implies that $\varphi$ is well-defined. In the same
way Propositions 5.2 and 5.3 imply that $\psi$ is well-defined.
Thus, it is enough to see that $\varphi$ and $\psi$ are each other's inverse.

First, let $I_i\in\mathcal{I}(A_i,\omega_i)$. Then, obviously $(I_1\bowtie I_2)_i=I_i$ and this shows that $\psi\varphi(I_1,I_2)=(I_1,I_2)$.

On the other hand, let $I\in\mathcal{I}(A_1\bowtie A_2)\setminus\{\textrm{Ker}\ \omega_1\bowtie\omega_2\}$. Clearly $I\subseteq I_1\bowtie I_2$. Conversely, let $(a,b)\in I_1\bowtie I_2$. By definition there exists $b'\in A_2$ such that $(a,b')\in I$. Since $w_1(a)=0$ it follows that $w_2(b')=0$ and, since $I$ is an ideal we have that $(a,0)=(1,0)(a,b')\in I$. In the same way $(0,b)\in I$ and we have that $I\subseteq I_1\bowtie I_2$; i.e., that $\varphi\psi(I)=I$ and the result follows.
\end{proof}

\begin{exa}
Let $K$ be any field. We construct $(K\bowtie K,\
\textrm{id}_K\bowtie\textrm{id}_K)$ like in Example 2.1. Then a direct
application of the previous proposition gives us the simplicity of
$\textrm{Ker}\ \textrm{id}_K\bowtie\textrm{id}_K$.
\end{exa}

In \cite{CG} the notion of decomposable baric algebra was introduced. Namely, a baric algebra $(A,\omega)$ with an idempotent of weight 1 is decomposable if there are non-trivial ideals $N_1$ and $N_2$ of $A$, both contained in $\textrm{Ker}\ \omega$ and such that $\textrm{Ker}\ \omega=N_1\oplus N_2$. Otherwise $(A,\omega)$ is indecomposable.

The following result shows that our construction works nicely with respect to indecomposability in the commutative case.

\begin{prop}
Let $(A_i,\omega_i)$ be commutative unital indecomposable baric algebras for $i=1,2$. Then $(A_1\bowtie A_2,\omega_1\bowtie\omega_2)$ is also indecomposable.
\end{prop}
\begin{proof}
Assume that $(A_1\bowtie A_2,\omega_1\bowtie\omega_2)$ is decomposable. Then there exist ideals $S,T$ such that $\textrm{Ker}\ \omega_1\bowtie\omega_2=S\oplus T$. Since both $S,T$ are non-trivial we can apply Proposition 5.4 to get that $S=(S_1,S_2)$ and $T=(T_1,T_2)$ with $S_i,T_i\unlhd A_i$.

Clearly $S_i+T_i\subseteq\textrm{Ker}\ \omega_i$. Now, if $x\in S_i\cap T_i$ it follows that $(x,0)\in S\cap T=0$ so $S_i$ and $T_i$ have direct sum. Moreover, since $\textrm{Ker}\ \omega_i\unlhd A_1\bowtie A_2$, $\textrm{Ker}\ \omega_1\cap T=\textrm{Ker}\ \omega_2\cap S=0$ it follows that $\textrm{Ker}\ \omega_i=S_i\oplus T_i$.

If $S_1=0$ then it must be $S_2\neq 0$. Moreover, $\textrm{Ker}\ \omega_1\subseteq T_1$ and if it was $T_2=0$ it follows that $\textrm{Ker}\ \omega_2\subseteq S_2$ and $\textrm{Ker}\ \omega_1\oplus\textrm{Ker}\ \omega_2=\textrm{Ker}\ \omega_1\bowtie\omega_2$ which is false by Proposition 5.4 again. Consequently we have proved that if $S_1=0$, then $S_2,T_2\neq 0$ and $(A_2,\omega_2)$ is decomposable.

In the same way it follows that $T_1=0$ implies that $(A_2,\omega_2)$ is decomposable.

If both $S_1$ and $T_1$ are non-zero, then $(A_1,\omega_1)$ is decomposable and the result follows.
\end{proof}

\section{Associativity}
We will start this section with the following lemma:

\begin{lem}
Let $(A_i,\omega_i)$ for $i=1,2$ be baric algebras. Let
$x=(a_1,a_2)$, $y=(b_1,b_2)$ and $z=(c_1,c_2)$ be elements of
$A_1\bowtie A_2$. Then we have that:
$$(x,y,z)=((a_1,b_1,c_1)+\omega_2(b_2)(a_1c_1-\omega_1(c_1)a_1),(a_2,b_2,c_2)+\omega_1(b_1)(a_2c_2-\omega_2(c_2)a_2)).$$
\end{lem}

We can use this to prove the following characterization:

\begin{prop}
Let $(A_i,\omega_i)$ with $i=1,2$ be baric algebras. Then the
algebra $A_1\bowtie A_2$ is associative if and only if
$(a_1,a_2)(b_1,b_2)=(\omega_1\bowtie\omega_2(b_1,b_2))(a_1,a_2)$ for
all $(a_1,a_2),(b_1,b_2)\in A_1\bowtie A_2$.
\end{prop}
\begin{proof}
Put $x=(a_1,a_2)$, $y=(b_1,b_2)$ and $z=(c_1,c_2)$. Let us suppose
that $A_1\bowtie A_2$ is associative. Then each $A_i$ is also
associative because they are subalgebras of $A_1\bowtie A_2$. So, by
Lemma 6.1:
$$0=(x,y,z)=(\omega_2(b_2)(a_1c_1-\omega_1(c_1)a_1),\omega_1(b_1)(a_2c_2-\omega_2(c_2)a_2)).$$
and choosing $b_i\notin \textrm{Ker}\ \omega_i$ we have that
$a_ic_i=\omega_i(c_i)a_i$ for all $a_i,c_i\in A_i$. Thus we have
that $(a_1,a_2)(b_1,b_2)(=\omega_1\bowtie\omega_2(b_1,b_2))(a_1,a_2)$
and the proof is complete as the converse is just an easy
computation.
\end{proof}

A $K$-algebra $A$ is called left (resp. right) alternative if
$(x,x,y)=0$ for all $x,y\in A$ (resp. $(x,y,y)=0$ for all $x,y\in
A$). We say that $A$ is alternative if it is both left and right
alternative. Of course an associative algebra is left and right
alternative. As an easy consequence of Lemma 6.1 we have:

\begin{prop}
Let $(A_i,\omega_i)$ with $i=1,2$ be associative baric algebras.
Then the following are equivalent:
\begin{itemize}
\item[(i)] $A_1\bowtie A_2$ is associative.
\item[(ii)] $A_1\bowtie A_2$ is left alternative.
\item[(iii)] $A_1\bowtie A_2$ is right alternative.
\end{itemize}
\end{prop}

\begin{exa}
Let $K$ be any field. Thanks to Proposition 3.1(ii) and recalling
Example 2.1  , we can unambiguously define the baric algebra
$(K^{\bowtie n},\textrm{id}_K^{\bowtie n})$, where $K^{\bowtie n}$
stands for $K\bowtie\dots\bowtie K$ and $\textrm{id}_K^{\bowtie
n}=\textrm{id}_K\bowtie\dots\bowtie\textrm{id}_K$ is defined by the
formula $\textrm{id}_K^{\bowtie
n}(\alpha_1,\dots,\alpha_n)=\alpha_1+\dots+\alpha_n$. Then, due to
Proposition 6.1, $K^{\bowtie n}$ is associative.
\end{exa}

The remaining of this section will be devoted to show that, under
certain assumptions, the previous example is the only situation in
which our construction is associative.

Let $(A,\omega)$ be a baric algebra over a field $K$ and let us
choose $\{e_i\ |\ i\in I\}$ any $K$-basis for $A$. Put
$\epsilon_i=\omega(e_i)$ for all $i\in I$ and observe that we can
suppose, without loss of generality, that $\epsilon_i\in\{0,1\}$ for
all $i\in I$. Moreover we have:

\begin{lem}
Let $K$ be a field with $\textrm{char}\ K= 0$ and let $(A,\omega)$
be a baric $K$-algebra of countable dimension. Then $A$ admits a
basis such that every element in the basis is of weight 1.
\end{lem}
\begin{proof}
Let $\{e_i\ |\ i\in I\}$ with $|I|\leq\aleph_0$ be a $K$-basis of
$A$. We can suppose that $I\subseteq\mathbb{N}$ and that
$\epsilon_1=1$. Now for each $n\in I$ we define
$\displaystyle{e'_n=\frac{1}{\sum_{j\leq n}\epsilon_j}\sum_{j\leq
n}e_j}$. Then $\{e'_i\ |\ i\in I\}$ is the desired basis.
\end{proof}

\begin{prop}
Let $K$ be a field with $\textrm{char}\ K$=0 and let $(A,\omega)$ be
a countable-dimensional baric $K$-algebra such that $xy=\omega(y)x$
for all $x,y\in A$. Then, if $\nu=\textrm{dim}_K\ A$, we have
$(A,\omega)\cong(K^{\bowtie\nu},\textrm{id}_K^{\bowtie\nu})$ as
baric algebras.
\end{prop}
\begin{proof}
We consider the $K$-basis of $A$ $\{e_i\ |\ i\in I\}$ with
$\nu=|I|\leq\aleph_0$ and $\omega(e_i)=1$ for all $i\in I$ given by
the previous lemma. We define
$(A_i,\omega_i)=(Ke_i,\omega|_{Ke_i})$. Obviously
$(A,\omega)=(\bowtie_{i=1}^{\nu}A_i,\bowtie_{i=1}^{\nu}\omega_i)$
and the proof is complete as
$(Ke_i,\omega|_{Ke_i})\cong(K,\textrm{id}_K)$ trivially.
\end{proof}

Finally, as a consequence of this proposition we obtain the
following:

\begin{coro}
Let $(A_i,\omega_i)$ for $i=1,2$ be countable-dimensional baric
$K$-algebras with $\textrm{char}\ K=0$. Then $A_1\bowtie A_2$ is
associative if and only if
$(A_i,\omega_i)\cong(K^{\bowtie\nu_i},\textrm{id}_K^{\bowtie\nu_i})$
with $\nu_i=\textrm{dim}_K\ A_i$. In particular, $(A_1\bowtie
A_2,\omega_1\bowtie\omega_2)\cong(K^{\bowtie\nu},\textrm{id}_K^{\bowtie\nu})$
with $\nu=\nu_1+\nu_2$.
\end{coro}

\end{document}